\documentclass[a4paper]{article}
\usepackage{a4wide}

\usepackage{graphicx}      

\usepackage{amsmath, amsthm, amssymb}
\usepackage{amsfonts, authblk}

\usepackage[usenames,dvipsnames]{xcolor}

\usepackage{accents}
\usepackage{hyperref}
\usepackage{cite}

\newtheorem{thm}{Theorem}
\newtheorem{cor}[thm]{Corollary}
\newtheorem{lem}[thm]{Lemma}

\newtheorem{assum}[thm]{Assumption}

\newtheorem{rem}[thm]{Remark}

\newcommand{\e}{\mathrm{e}}
\newcommand{\diff}{\mathrm{d}}
\newcommand{\D}{\mathcal{D}}
\newcommand{\K}{\mathcal{K}}
\newcommand{\tobs}{t_{\mathrm{obs}}}
\newcommand{\tstab}{t_{\mathrm{stab}}}
\newcommand{\tobsmax}{\bar{t}_\mathrm{obs}}

\newcommand{\gmin}{g_{\min}}
\renewcommand{\epsilon}{\varepsilon}
\newcommand{\eps}{\epsilon}
\newcommand{\R}{\mathbb{R}}
\newcommand{\xhat}{\hat{x}}
\renewcommand{\geq}{\geqslant}
\renewcommand{\leq}{\leqslant}
\newcommand{\kmax}{k_{\max}}
\newcommand{\N}{\mathbb{N}}
\newcommand{\gram}{G}
\newcommand{\smax}{\bar{s}}
\newcommand{\smin}{\ubar{s}}
\DeclareMathOperator{\Id}{Id}
\DeclareMathOperator{\tr}{tr}
\DeclareMathOperator{\diam}{diam}
\DeclareMathOperator{\dist}{dist}
\newcommand{\sym}{S^n_{++}}
\newcommand{\Khat}{\widehat{\K}}
\renewcommand{\S}{\mathcal{S}}
\newcommand{\ubar}[1]{\underaccent{\bar}{#1}}

\title{Output feedback stabilization of non-uniformly
observable systems by means of a switched Kalman-like observer}
\author[1]{Lucas Brivadis}
\author[2]{Ludovic Sacchelli}

\affil[1]{\small Université Paris-Saclay, CNRS, CentraleSupélec, Laboratoire des Signaux et Systèmes, 91190, Gif-sur-Yvette, France. (email: lucas.brivadis@centralesupelec.fr)}
\affil[2]{\small Inria, Université Côte d’Azur, CNRS, LJAD, MCTAO team, Sophia Antipolis, France.
 (email: {ludovic.sacchelli@inria.fr})}

\begin{document}

\maketitle



\begin{abstract}                
We propose to explore switching methods in order to recover some properties of Kalman-like observers for output feedback stabilization of state-affine systems that may present observability singularities. The self-tuning gain matrix in Kalman-like observers tend to be singular in the case of non-uniformly observable systems. We show in the case of state-affine systems with observable target that it can be prevented by dynamically monitoring observability of the system, and switching the control when it becomes critical. 
\end{abstract}



\section{Introduction}

Coupling
a stabilizing state feedback and an observer is a tried and tested method for stabilization of systems whose state may be only partially known \cite{AndrieuPraly2009}. A major issue in designing that coupling in the context of nonlinear autonomous systems is that observability of the system (hence, the ability of the observer to estimate the state), may vary depending on the control. It was shown in the 90' (see \cite{TeelPraly1994, jouan1996finite}) that under assumption of observability for any control, a separation principle could be obtained. Numerous works have been dedicated to lifting this assumption. Early on, it was shown in \cite{MR1359027} that allowing a time-periodic feedback law is sufficient to obtain local output feedback stabilization of non-uniformly observable systems. Later on, it was shown in \cite{MR2137765} that the existence of just one observable control was sufficient to achieve semi-global practical stabilization, that is, stabilization in any arbitrary small neighborhood of the target point. Both of these papers rely on a natural idea of dealing with estimation and stabilisation in two distinct alternating modes of a switching observer.
Essentially, \cite{MR2137765} relies on a periodic switching strategy of the control between observation modes (where the input makes the system observable and the observer converges) and stabilization modes (where the input steers the state to the target if the observer is sufficiently close to the state).
Over the years, this switching/hybrid strategy has been developed in various contexts of non-uniformly observable systems \cite{nesic, 9683664, Besanon2002}, and applied for example to anti-lock braking systems (ABS) \cite{9363594, preprintMPLLAR}.

In this paper, we focus on state-affine systems that are observable at the target.
Our goal is to rely on a Kalman-like observer with dynamic gain following a Lyapunov differential equation 
to estimate the state,  ensure semi-global asymptotic stability, and guarantee that if the system is uniformly observable then our strategy coincides with the usual nonlinear separation principle \cite{MR1280224}.
A major difficulty however, is that the pivotal properties of the dynamic gain matrix do not mesh well with non-uniform observability \cite{MR4252728}. Essentially, the dynamic gain can become singular when the system is unobservable, which may happen for the system in closed loop. Our goal is to propose a solution to avoid singularization of the gain matrix, 
but without requiring prior knowledge of singular inputs.
We 
tackle this issue by an event-triggered law guaranteeing lower boundedness of the Gram observability matrix via its online monitoring.

We apply the following procedure. \textit{(i)} For a fixed time, we apply to the system a null input. The system is observable, and we estimate the state with an observer whose dynamic gain follows a Lyapunov differential equation. Since the target remains an equilibrium under the null input, these \textit{observation modes} do not prevent stability. \textit{(ii)} After this mode, the system is in \textit{stabilization mode}, where the control is chosen as the stabilizing feedback evaluated on the observer. The system remains in stabilization mode for another fixed time, in order to guarantee the decrease of a Lyapunov function over an observation--stabilization cycle. However we do not stop the stabilization mode (and go back to \textit{(i)}) until the system's observability becomes too critical, a fact we measure thanks to online computation of the observability Gramian. In particular, the system will stop switching once it is sufficiently close to the target, allowing exact convergence.

\section{Problem statement}

Let $n\in \N$ be a positive integer, let $A:\R\to \R^{n\times n}$, and  $B:\R\to \R^{n\times 1}$, be  locally Lipschitz  maps valued in the set of $n\times n$, respectively $n\times 1$, real valued matrices. Let $C\in \R^{1\times n}$ be a linear form.
For any $u\in L^{\infty}([0,+\infty),\R)$, we consider the following single-input single-output (SISO) state-affine system
\begin{equation}\label{eq:syst}
\left\{
\begin{aligned}
&\dot x(t) = A(u(t))x(t) + B(u(t))\\
&y(t) = Cx(t).
\end{aligned}
\right.
\end{equation}
A particular instance of this dynamical model is the classical SISO bilinear case where $A(u) = A_0 + uA_1$, $B(u) = u B_1$, with $A_0,A_1\in \R^{n\times n}$ and $B_1\in \R^{n\times 1}$.

System~\eqref{eq:syst} is said to be \emph{observable for the control $u$ over $[t_0,t_1]$} if, for all pairs of solutions $((x_a,y_a),(x_b,y_b))$ of \eqref{eq:syst}, $\left.y_a\right|_{[t_0,t_1]}\equiv \left.y_b\right|_{[t_0,t_1]}$ implies $\left.x_a\right|_{[t_0,t_1]}\equiv \left.x_b\right|_{[t_0,t_1]}$. If the system is observable for any control over any time interval, then it is said to be \emph{uniformly observable in small time}. 
In this paper we only require observability at the target, i.e., observability of the system for the constant input $u=0$.
\begin{assum}\label{ass:0_observable}
The pair $(C, A(0))$ is observable.
\end{assum}

In the particular case of of bilinear systems, where $A(u)=A_0+u A_1$, the observability of the pair $(C, A_0)$ is satisfied over an open and dense subset of all possible pairs $(C, A_0)$. This contrasts with the existence of controls for which the system is unobservable, a property satisfied by a residual subset of the set of possible triples $(C,A_0,A_1)$ (see, for instance, \cite{brivadis:hal-03589760}).

We wish to study semi-global output feedback stabilization of \eqref{eq:syst} at $0\in\R^n$. To do so, we first assume that state feedback stabilization is achievable by means of a state feedback.


\begin{assum}\label{ass:stab}
There exists a locally Lipschitz bounded feedback law $\lambda:\R^n\to \R$ such that $0\in\R^n$ is a locally asymptotically stable equilibrium point of
the vector field $\R^n\ni x\mapsto A(\lambda(x))x+B(\lambda(x))$.
\end{assum}

Under Assumption~\ref{ass:stab} we denote by $\D$ the basin of attraction of the origin and set $\bar u := \sup_{\R^n}|\lambda|<+\infty$.
According to the converse Lyapunov theorem \cite[Theorem 2.296]{PralyBresch2022a} (see also \cite{MR1765429}, based on the previous works of \cite{kurzweiloriginal, kurzweil, massera}),
there exists a proper function $V\in C^\infty(\D, \R_+)$ such that $V(0)=0$ and
\begin{equation}\label{eq:lyap}
    \frac{\partial V}{\partial x}(x)(A(\lambda(x))x+B(\lambda(x))) \leq - V(x),
    \quad 
    \forall x\in\D.
\end{equation}

\begin{rem}
In the context of semi-global stabilization, the boundedness requirement in Assumption~\ref{ass:stab} is easy to fulfill. Indeed, if
there exists an unbounded smooth feedback law $\tilde \lambda$ such that $0$ is an equilibrium of
$x\mapsto A(\lambda(x))x+B(\lambda(x))$ that is locally asymptotically stable with basin of attraction $\tilde \D$,
then for any compact set $\K\subset\tilde \D$, and any bounded smooth feedback law $\lambda$ defined by $\lambda(x) = \tilde \lambda(x)$ for all $x\in\K$ and $\lambda(x) = 0$ for all $x\in\R^n\setminus\tilde \D$,
we have that $\lambda$ satisfies Assumption~\ref{ass:stab} with a basin of attraction containing $\K$.
In other words, one can always construct a bounded smooth stabilizing state feedback from an unbounded one, up to a reduction of the basin of attraction. In particular, note that if $\tilde \D = \R^n$ (i.e. there exists a smooth globally asymptotically stabilizing state feedback), then for any compact set $\K\subset\R^n$ this procedure allows to construct a bounded smooth locally asymptotically stabilizing state feedback with basin of attraction containing $\K$ (i.e. to achieve semi-global asymptotic state feedback stabilization).
\end{rem}

On the considered class of systems, we focus on Kalman-like observers.
Denoting by $\varepsilon=\xhat-x$ the estimation error, for a given $u\in L^{\infty}([0,+\infty),\R)$ and a given positive parameter $\theta$, we consider an observer given by
\begin{align}
    \dot \xhat(t)& = A(u(t))\xhat(t) + B(u(t)) - S(t)^{-1}C'(C\xhat(t)-y(t)),
    \label{eq:obs}
    \\
    \dot \eps(t)& = A(u(t))\eps(t) - S(t)^{-1}C'C\eps(t),
    \label{eq:eps}
    \\
    \dot S(t)&= -A(u(t))' S(t) -S(t) A(u(t)) -\theta(t) S(t) +C'C.
    \label{eq:S}
\end{align}
evolving on $\R^n\times\R^n\times\sym$, where $\sym$ denotes the set of positive definite matrices in $\R^{n\times n}$. Here, $S$ is a dynamic gain matrix following a Lyapunov differential equation \cite{MR1997753,MR1343974}.

We wish to follow a classical state-observer coupling in order to achieve stabilization in the output feedback case.
In the case of uniformly observable systems, it is well-known that semi-global dynamic output feedback stabilization can be achieved by choosing $u=\lambda(\xhat)$ and $\theta$ large enough, see e.g., \cite{MR1280224}.
Here we focus on the case of systems which are not uniformly observable, i.e., for which there exists a bounded input $u$ making \eqref{eq:syst} unobservable.
To achieve output feedback stabilization in that case, we propose a novel strategy based on switches in the control law and in the dynamics of the observer.

\section{Switching strategy and main result}



Let $\tobs$, $\tstab$, $\alpha$, $\beta$ and $\gmin$ be positive constants. Let $T\in(0, \tstab)$.
For all $(\xhat_0, \eps_0, S_0)\in
\R^n\times\R^n\times\sym
$,
let us construct according to the procedure below a sequence of switching times $(t_k)_{0\leq k< \kmax}$ in $\R_+\cup\{+\infty\}$ for some $k\in\N$ such that $\kmax\geq 3$,
and a continuous trajectory $(\xhat, \eps, S):\R_+\to\R^n\times\R^n\times\sym$ starting at $(\xhat_0, \eps_0, S_0)$ satisfying some dynamics that switch at each $t_k$.
The sequence $(t_k)$ is defined such that $t_0=0$, $t_{k+1}-t_k\geq \max(\tobs, \tstab)$ and either $\kmax=+\infty$, or $\kmax$ is finite and $t_{\kmax-1}=+\infty$. In that way, $([t_k, t_{k+1}))_{0\leq k< \kmax}$ is a partition of $\R_+$.

Set $t_0=0$.
Assume that $t_{2k}\geq 0$ is defined for some $k\in\N$ and is finite.
Set $t_{2k+1}=t_{2k}+\tobs$.
Over $[t_{2k}, t_{2k+1})$, define the Cauchy problem
\begin{equation}
\left\{
\begin{aligned}
    &\dot \xhat = A(0)\xhat - S^{-1}C'(C\xhat-y)
    \\
    &\dot \eps = A(0)\eps - S^{-1}C'C\eps
    \\
    &\dot S= -A(0)' S -S A(0) -\alpha S +C'C
\end{aligned}
\right.
\end{equation}
initialized at $t_{2k}$ by
$(\xhat,\eps,S)(t_{2k})
=
\begin{cases}
(\xhat_0,\eps_0,S_0)&\text{if }k=0\\
(\xhat,\eps,S)(t_{2k}^-)&\text{otherwise}
\end{cases}
$.
This system admits a unique global solution according to Cauchy--Lipschitz theorem since $\lambda$ is bounded and $S(t)$ is lower bounded (see \eqref{E:bound_below_S_obs}).
Then, on $[t_{2k+1},+\infty)$, we define the Cauchy problem 
\begin{equation}\label{E:system_u=lambda}
\left\{
\begin{aligned}
    &\dot \xhat = A(\lambda(\xhat))\xhat + B(\lambda(\xhat)) - S^{-1}C'(C\xhat-y)
    \\
    &\dot \eps = A(\lambda(\xhat))\eps - S^{-1}C'C\eps
    \\
    &\dot S= -A(\lambda(\xhat))' S -S A(\lambda(\xhat)) -\beta S +C'C
\end{aligned}
\right.
\end{equation}
initialized at $t_{2k+1}$ by
$$
(\xhat,\eps,S)(t_{2k+1})=
(\xhat,\eps,S)(t_{2k+1}^-)
$$
Using Grönwall's inequality, one can easily show that this system admits a unique global solution according to Cauchy--Lipschitz theorem since $\lambda$ is bounded and $S(t)$ is lower bounded by an exponentially decreasing function (see \eqref{E:bound_below_S_unobs}).
Then we define the next switching time $t_{2k+2}$ by
$$
t_{2k+2} = \inf\{t>t_{2k+1}+\tstab\mid G_{\lambda\circ\xhat}(t-T, t)\not >\gmin \Id\},
$$
i.e. $t_{2k+2}$ is the smallest time larger than $t_{2k+1}+\tstab$ such that the lowest eigenvalue of $G_{\lambda\circ\xhat}(t-T, t)$ is smaller than $\gmin$.
If $t_{2k+2} = +\infty$, then there are no more switches and $\kmax = 2k+3$.
This concludes the inductive construction of $(t_k)$ and $(S, \xhat, \eps)$.
To summarize, $[t_{2k}, t_{2k+1})$ are observation modes, while $[t_{2k+1}, t_{2k+2})$ are stabilization modes.
While observation modes have constant length $\tobs$, stabilization modes last at least $\tstab$ and as long as the observability of the system is considered to be sufficient (when compared to $\gmin$). This approach can be compared with the multi-observer approach where the observer dynamics switch between different modes depending on a criterion (see, e.g., \cite{petri:hal-03766736}). However, contrary to \cite{petri:hal-03766736}, our objective in doing so is not to increase speed of convergence but rather to robustify the observer with respect to observability singularities.

The resulting system can be written as:
\begin{equation}\label{eq:switch}
\left\{
\begin{aligned}
    &\dot x = A(u(t))\xhat + B(u(t))
    \\
    &\dot \xhat = A(u(t))\xhat + B(u(t)) - S^{-1}C'(C\xhat-Cx)
    \\
    &\dot S= -A(u(t))' S -S A(u(t)) -\theta(t) S +C'C,
\end{aligned}
\right.
\end{equation}
where 
$(x, \xhat, S)$ lies in $\R^n\times\R^n\times\sym$
and
$$(u(t), \theta(t)) = \begin{cases}
(0, \alpha) &\text{if } t_{2k}\leq t<t_{2k+1} \text{ for some }k\\
(\lambda(\xhat(t)), \beta) &\text{if } t_{2k+1}\leq t<t_{2k+2} \text{ for some }k\\
\end{cases}.$$


Now we state our main result on the semi-global output feedback stabilization of \eqref{eq:syst}.

\begin{thm}\label{th:main}
Suppose Assumptions~\ref{ass:0_observable} and \ref{ass:stab} hold.
For all compact set $\K\times\Khat\times\S\subset\D\times\R^n\times\sym$,
there exist positive constants $\tstab$, $\tobs$, $T$, $\gmin$, $\beta$ and $\alpha$ such that
the closed-loop system \eqref{eq:switch} is such that:
\begin{itemize}
    \item There is at most a finite number of switches, i.e., $\kmax<+\infty$.
    \item For all $(x_0, \xhat_0, S_0)\in \K\times\Khat\times\S$, the corresponding trajectory $(x, \xhat, S)$ of the closed-loop system \eqref{eq:switch} is such that $(x(t), \xhat(t))$ tend towards $(0, 0)$ as $t$ goes to infinity. Moreover, $S(t)$ tends toward the unique solution $S_\infty\in\sym$ of the Lyapunov equation $A(0)'S+SA(0)+\beta S = C'C$ and $S$ remains upper and lower bounded over $\R_+$.
    \item For all $R>0$, there exists $r>0$ such that for all $\tau\in\R_+$, if $|x(\tau)|<r$ and $|\xhat(\tau)|<r$,
    then $|x(t)|<R$ and $|\xhat(t)|<R$ for all $t\geq \tau$.
\end{itemize}

Moreover, the positive constants can be taken according to the following procedure:
for all $\tstab>0$,
there exists $\tobsmax>0$,
such that for all 
$\tobs\in(0,\tobsmax)$,
there exists
$\bar{T}>0$,
such that for all 
$T\in(0,\bar{T})$,
there exists
$\bar{g}_{\min}>0$,
such that for all $\gmin\in(0, \bar{g}_{\min})$
there exists
$\ubar{\beta}>0$,
such that for all $\beta>\ubar{\beta}$
there exists
$\ubar{\alpha}>0$,
such that for all 
$\alpha>\ubar{\alpha}$,
the result holds.
\end{thm}

\begin{rem}

\begin{itemize}
    \item Theorem~\ref{th:main} is a semi-global output feedback stabilization strategy for state-affine systems that are observable at the target. It can be interpreted as a nonlinear separation principle for this class of non-uniformly observable systems. Note however that, as usual for nonlinear systems \cite{AndrieuPraly2009}, the observer and the feedback law cannot be designed separately.

    \item Regarding stability of the closed-loop, note that the system is ony stable with respect to variables $x$ and $\xhat$, but not with respect to $S$, despite the attractivity of $S_\infty$. In other words, the system is only stable on the set $\{0\}\times\{0\}\times\sym$. This is due to the switching strategy, that prevent $S_\infty$ to be an equilibrium point during observation modes (since the gain of the observer is switched from $\beta$ to $\alpha$). Similarly, stability with respect to variables used in the switching condition (namely, $G_u(t-T, t)$, that can be integrated as a state variable thanks to \eqref{eq:res}-\eqref{eq:gram}) is not investigated, and should be tackled in future works.

    \item Observability for a given control on a time interval is an open condition. As a result, observability for the null control implies observability of small enough controls (on a given time frame). As is shown in the proof of Theorem~\ref{th:main}, this can be leveraged to prove that the system switches at most a finite number of times along a given trajectory. 
    It should be noted that in the case where the system is uniformly observable, we then recover the usual non-switching strategy for a separation principle (i.e. $u = \lambda(\xhat)$ and $\beta$ large enough), at least after the preliminary observation phase.

\end{itemize}

\end{rem}


The proof of this theorem is the result of a sequence of lemmas exposed in Section~\ref{S:proof}. We first show boundedness of the trajectories, which then helps prove convergence of the state and observer to the target, and finally discuss stability at the target.
Below, we recall some important properties of Lyapunov differential equations and observability Gramian.

\section{Lyapunov differential equation and Gram observability matrix}
For any bounded $u:[t_0,t_1]\to \mathbb R$, let $\Phi_u:\mathbb [t_0,t_1]^2\to \mathrm{GL}_n(\mathbb{R})$ be the state transition matrix such that
$$
\frac{\partial }{\partial t} \Phi_u(t,s) = -\Phi_u(t,s)A(u(t)), \quad \Phi_u(s,s)=\mathrm{Id}.
$$
We also have $\dfrac{\partial }{\partial s} \Phi_u(t,s) =A(u(s))\Phi_u(t,s) $ and
\begin{equation}\label{eq:res}
    \frac{\diff}{\diff t}\Phi_u(t, t-T) = - \Phi_u(t, t-T)A(u(t)) + A(u(t-T))\Phi_u(t, t-T).
\end{equation}

The \emph{Gram observability matrix}, or \emph{observability Gramian matrix}, for the control $u$ over the time interval $[t_0,t_1]$ is defined as
\begin{equation}
    \int_{t_0}^{t_1} \Phi_u(t_1,s)'C'C\Phi_u(t_1,s)\diff s=G_u(t_0,t_1).
\end{equation}
The Gram observability matrix contains some measure of the observability of a control $u$ by linking it to the positive-definiteness of $G_u(t_0,t_1)$. Indeed for any $\omega:[t_0,t_1]\to \R^n$ such that $\dot\omega(t)=A(u(t))\omega$,
$$
\int_{t_0}^{t_1} |C \omega(s)|^2\diff s= \omega(t_1)' G_u(t_0,t_1)\omega(t_1).
$$
Inobservability of the control $u$ over $[t_0,t_1]$ implies the existence of a nontrivial kernel for $G_u(t_0,t_1)$.
For constant inputs $u$, observability of the pair $(C, A(u))$ is equivalent to the positive-definiteness of $G_u(t_0,t_1)$ for any $t_1>t_0>0$.

On the set of symmetric matrices, we consider the Lyapunov differential equation with gain $\theta>0$
\begin{equation}\label{E:LyapEq}
\begin{aligned}
\dot S(t)&= -A(u(t))' S(t) -S(t) A(u(t)) -\theta S(t) +C'C,
\\
S(t_0)&=S_0\in \sym.
\end{aligned}
\end{equation}
The solution to this differential equation admits an explicit variation of constants type expression (see, for instance, \cite[Theorem 1.1.5]{MR1997753}):
\begin{equation} \label{E:VariationConst}
S(t)=
\e^{-\theta(t-t_0)}\Phi_u(t,t_0)'S(t_0)\Phi_u(t,t_0)
+
\int_{t_0}^{t} \e^{-\theta(t-s)}\Phi_u(t,s)'C'C\Phi_u(t,s)\diff s.
\end{equation}
In particular, the Gramian $G_u(t_0,t_1)$ is the evaluation at time $t_1$ of the solution of \eqref{E:VariationConst} with gain $\theta=0$ and initial condition $S(t_0)=0$.

Under the assumption that $S(t_0)> 0$, $S(t)>0$ for all $t\geq t_0$. Each member in the right-hand side of  \eqref{E:VariationConst} can bring competing lower bounds of $S$, depending on the context. One relies on the Gramian matrix, useful under observability assumptions
\begin{equation}\label{E:bound_below_S_obs}
    S(t)\geq \e^{-\theta(t-t_0)}\int_{t_0}^{t} \Phi_u(t,s)'C'C\Phi_u(t,s)\diff s=\e^{-\theta(t-t_0)}G_u(t_0,t),
\end{equation}
while the other leads to a worst case scenario lower bound, with $a_\infty=\sup_{s\in(t_0,t)}\|A(u(s))\|$:
\begin{equation}\label{E:bound_below_S_unobs}
    S(t)\geq \e^{-\theta(t-t_0)}\Phi_u(t,t_0)'S(t_0)\Phi_u(t,t_0)\geq \e^{-(\theta+2a_\infty)(t-t_0)}S_{\min}(t_0)\Id.
\end{equation}

These bounds can then be used for computations of a Lyapunov function for the error in equations \eqref{eq:obs}-\eqref{eq:S}: $\eps'S\eps$. Indeed, with constant gain $\theta$, 
$
    \frac{\diff}{\diff t}\eps'S\eps
    = -\theta\eps'S\eps - \eps'C'C\eps
    \leq -\theta\eps'S\eps.
$
Then for all $t\geq t_0$, $\eps'S\eps(t) \leq \e^{-\theta(t-t_0)}\eps'S\eps(t_{0})$, which translates to the crucial error bound
\begin{equation}\label{E:error_bound}
    |\eps(t)| \leq \e^{-\frac{\theta}{2}(t-t_0)} \sqrt{\frac{S_{\max}(t_{0})}{S_{\min}(t)}}|\eps(t_{0})|.
\end{equation}
Here $S_{\min}$ and $S_{\max}$ respectively denote the smallest and largest eigenvalues of the positive-definite matrix $S$.
Under observability assumption, it may then be worthwhile to bound $S_{\min}(t)$ by below using \eqref{E:bound_below_S_obs}, as $S_{\min}(t)\geq \e^{-\theta T}g(t)$, where $g(t)$ is the smallest eigenvalue of $G_u(t-T,t)$. 
On the computational side, we can use a Lyapunov differential style equation to compute the Gramian over a sliding interval of length $T$.
For all $t\geq T$,
\begin{equation}\label{eq:gram}
\frac{\diff}{\diff t}G_u(t-T, t) = - A(u(t))'G_u(t-T, t) - G_u(t-T, t)A(u(t)) + C'C -\Phi_u(t, t-T)'C'C\Phi_u(t, t-T).
\end{equation}
Regarding $S_{\max}$, we show that it is bounded provided $\theta$ is large enough. 
The quantity $\tr(S)$ satisfies for a given gain $\theta>0$
$$
\frac{\diff \tr(S)}{\diff t}=-2 \tr ( A' S)-\theta \tr(S)+\tr(C'C).
$$
Since we have $\tr ( A' S) \leq \sqrt{\tr (A'A)}\sqrt{\tr (S^2)}$, $\sqrt{\tr (S^2)}\leq \tr(S)$ and $\tr (C'C)=|C|^2$,
$$
\frac{\diff \tr (S) }{\diff t}\leq (-\theta+2 a_F) \tr (S) +\tr( C'C)
$$
with $a_F=\sup_{[t_0,t_1]}\sqrt{\tr (A'(u(t))A(u(t)))}$.
By Gr\"onwall's inequality,
$$
\tr(S(t))\leq \left(\tr(S(t_0))+|C|^2 (t-t_0)\right) \e^{-(\theta-2 a_F)(t-t_0)}.
$$
As a consequence, as soon as $\theta>2 a_F$,
\begin{equation}
    \label{E:majS}
    S_{\max}(t)\leq
    \tr(S(t)) \leq
    \max
    \left(
    \tr(S(t_0))
    ,
    \frac{|C|^2}{\theta-2 a_F}
    \right).
\end{equation}









\section{Proof of Theorem~\ref{th:main}}
\label{S:proof}
\subsection{Preliminaries and notations}\label{sec:prel}

For all $R>0$, we denote by $\D(R) = \{x\in\R^n\mid V(x)\leq R\}$ and by
$m(R)=\displaystyle\sup_{\D(R)}\left|
    \dfrac{\partial V}{\partial x}
\right|$
where $V$ is the Lyapunov function given by \eqref{eq:lyap}.
In particular, let $R_0>0$ be such that $\D(R_0)$ contains $\K$.
For all $R>0$, we have
$
m(R)
\diam \D(R)
\geq R
$. Indeed, according to the mean value theorem, we get
\begin{align}\label{E:dVdiam}
m(R)
\diam \D(R)
\geq
\sup_{x\in\D(R)} m(R) |x|
\geq \sup_{x\in\D(R)} V(x) = R.
\end{align}

In these preliminaries, we choose successively $\tstab$, $\tobs$, $T$ and $\gmin$.
Let $\tstab>0$ be fixed. Define
\begin{equation}\label{E:def_eta}
    \eta=
    \frac{
        \e^{\tstab}-1-\tstab
    }{2+\e^{\tstab}}>0.
\end{equation}
Set $a_0=\|A(0)\|$,
$
        a_\infty
        =
        \sup_{x\in \R^n} \| A(\lambda(x))\|
$
and define
\begin{equation}\label{E:def_tobsmax}
    \tobsmax 
    =
    \sup 
    \left\{
        t\geq 0: 
        \sup_{R\in (0,R_0]}
        \frac{m((1+\eta)R)}{R}
\left( a_0\diam \D(R) +R\right) t\e^{ a_0 t} <\eta
    \right\}.
\end{equation}
Equation~\eqref{E:dVdiam} implies that $\tobsmax\in(0, +\infty)$.
We pick $\tobs\in(0, \tobsmax)$ and $T<\bar{T}:=\min(\tobs,\tstab)/3$.
Then we pick $\bar{g}_{\min}>0$ such that 
$G_u(0,T)>\bar{g}_{\min}\Id$ for any $u = \lambda\circ \xhat$ such that $V(\xhat)\leq R_\infty$ and choose $\gmin\in(0, \bar{g}_{\min})$. We denote by $g_0$ the smallest eigenvalue of $G_0(0,T)$. Clearly, $\bar{g}_{\min}<g_0$ since the equilibrium trajectory $\xhat\equiv 0 $ leads to $\lambda\circ \xhat=0$. For all $t\geq T$, we denote by $g(t)$ the smallest eigenvalue of $G_u(t-T,t)$ for the control $u$ set in \eqref{eq:switch}, so that $G_u(t-T,t)>\gmin\Id$ and $G_u(t-T,t)\not>\gmin \Id$ can be shortened to $g(t)>\gmin$ and $g(t)\leq \gmin$, respectively.

In the following section, we discuss boundedness of the trajectories, which imply the possible choices for  $\ubar{\beta}$ and $\ubar{\alpha}$. With
$a_F=\sup_{x\in \R^n}\sqrt{ \tr(A(\lambda(x))'A(\lambda(x)))}$, we assume $\ubar{\beta}$ and   $\ubar{\alpha}$ large enough so that 
$\sup_{\S} \tr(S)>\dfrac{|C|^2}{\min(\ubar{\alpha},\ubar{\beta})-2 a_F}$. This allows to assume, in conjunction with \eqref{E:majS}, the existence of $\smax>0$ independent of $\alpha,\beta$, such that  all trajectories of \eqref{eq:switch} starting in $\K\times\Khat\times\S$ have $S_{\max}(t)<\smax$ for all $t\in \R_+$.
Let $(x_0, \xhat_0, S)$ be in $\K\times\Khat\times\S$.
Let us investigate the corresponding trajectory $(x, \xhat, S)$.

\subsection{Trajectories are bounded}


Recall that $t_{2k+1}=t_{2k}+\tobs$, $t_{2k+2}\geq t_{2k+1}+\tstab$. 
The first, and most technical, step of the boundedness proof is to show that over a sequence of switches on the time interval $[t_{2k},{t_{2k+1}+\tstab})$, the system remains bounded. We show that the Lyapunov for the system can grow in observability mode, but up to a tuning of the parameters, the stabilisation mode will compensate for that growth and ensure that the value of the Lyapunov at $t_{2k+1}+\tstab$ did not worsen from its value at $t_{2k}$.

\begin{lem}
\label{L:borne_alpha}
We define positive constants $K_1, K_2$, depending only on the problem data, $\K\times\Khat\times\S$ and $\tstab$, as follows:
with $d_0=\diam \left(\D( R_0)\cup \Khat\right)$, $D=\dist\big(\D((1+\eta)R_0),\D((1+2\eta)R_0)^c\big)$, $D'=\dist\big(\D((1-\eta)R_0),\D(R_0)^c\big)$, and $\bar m=m((1+2\eta)R_0)$
$$
K_1=\frac{\min\left(D,D'\e^{-a_\infty \tstab} \right)^2}{\smax d_0^2}
\quad \text{ and } \quad
K_2=\frac{ R_0^2\e^{-6a_\infty \tstab} }{\smax d_0^2 |C|^4  \bar m^2}.
$$
Assume $\alpha$ large enough so that
\begin{equation}\label{E:min_alpha_0}
    \e^{-\alpha(\tobs-T)}<K_1 g_0
\quad
\text{ and }
\quad 
\e^{-\alpha(\tobs-3T)}<K_2 g_0^3 \, \e^{-2\beta \tstab},
\end{equation}
then for all integers $k$ such that $2k\in [0,\kmax)$, if $V(x(t_{2k}))\leq R_0$, and either $V(\xhat(t_{2k}))\leq (1-\eta)R_0$ if $k\geq 1$ or $\xhat(0)\in \Khat$ if $k=0$, we have
$$
V(x(t_{2k+1}))\leq  (1+\eta)R_0, 
\qquad\quad 
V(\xhat(t_{2k+1}))\leq (1+2\eta)R_0,
$$
$$
V(x(t_{2k+1}+\tstab)) \leq R_0,
\quad \text{ and } \quad
V(\xhat(t_{2k+1}+\tstab))\leq (1-\eta)R_0.
$$

\end{lem}

\begin{proof}
{\it Step 1: observation mode.}
Let us bound 
$|x(t) - x(t_{2k})| $ on $[t_{2k},t_{2k}+\tobs)$. For $t>t_{2k}$,
\[
\begin{aligned}
    |x(t)-x(t_{2k})|
    &
    \leq \int_{t_{2k}}^t |\dot{x}(s)|\diff s
    \leq \int_{t_{2k}}^t |A(0) x(s) |\diff s
    \leq 
    a_0\int_{t_{2k}}^t |x(s)|\diff s 
    \\
    &\leq a_0\int_{t_{2k}}^t |x(s) - x(t_{2k})|\diff s +  
    a_0 |x(t_{2k})|
    (t-t_{2k})
    \\
    &\leq a_0\int_{t_{2k}}^t |x(s) - x_0|\diff s + a_0 \diam \D(R_0)  (t-t_{2k})
\end{aligned}
\]
By Grönwall's inequality, we conclude that 
\[
    |x(t)-x(t_{2k})|\leq \diam \D(R_0) a_0  (t-t_{2k})\e^{ a_0 (t-t_{2k})}.
\]
Assume there exists $t^*=\inf\{t\in [t_{2k},t_{2k}+\tobs):V(x(t^*))=(1+\eta)R_0\}$. By mean value inequality 
\[
    |V(x(t^*)) - V(x(t_{2k}))|
    \leq
    m((1+\eta)R_0)
    |x(t^*)-x(t_{2k})|
\]
and since $t^*<\tobs<\tobsmax$ (see \eqref{E:def_tobsmax}),
\[
    V(x(t^*))
    \leq  R_0 + 
m((1+\eta)R_0)
\diam \D(R_0)  a_0 (t^*-t_{2k})\e^{ a_0 (t^*-t_{2k})} <(1+\eta) R_0.
\]
This proves that $t^*$cannot exist and $V(x(t))\leq(1+\eta)R_0$ for all $t\in [t_{2k},t_{2k}+\tobs]$.


Regarding the followup assertions, we first look at the upper bound on $\eps=\xhat-x$ \eqref{E:error_bound}:
\begin{equation}\label{E:maj_eps_t2k+1+tobs}
|\eps(t_{2k}+\tobs)|
\leq 
\e^{-\frac{\alpha}{2} \tobs}
\sqrt{\frac{S_{\max}(t_{2k})}{S_{\min}(t_{2k}+\tobs)}} |\eps(t_{2k})|.
\end{equation}
With the assumption that 
$x(t_{2k})\in \D(R_0)$, $\xhat(t_{2k})\in \D((1-\eta)R_0)\cup \Khat\subset \D(R_0)\cup \Khat$, we have that $|\eps(t_{2k})|\leq  \diam \left(\D(R_0)\cup \Khat\right)=d_0$.
Furthermore, we know that $S_{\max}(t)<\smax$, and since $u=0$ on $[t_{2k},t_{2k}+\tobs)$,   $S_{\min}(_{2k}+\tobs)\geq \e^{-\alpha T}g_0$. Then
\begin{equation}\label{E:maj_eps_t2k+1+tobs_2}
|\eps(t_{2k}+\tobs)|
\leq 
\e^{-\frac{\alpha}{2} (\tobs-T)}
\sqrt{\frac{\smax}{g_0}} d_0
\end{equation}
Then by assuming that \eqref{E:min_alpha_0} holds, we have $|\eps(t_{2k}+\tobs)| \leq \dist\big(\D((1+\eta)R_0),\D((1+2\eta)R_0)^c\big)$ and thus
$
V(\xhat(t_{2k}+\tobs))\leq (1+2\eta)R_0.
$

{\it Step 2: stabilization mode.}
For all $t\in[t_{2k+1},t_{2k+1}+\tstab)$,
$$
\begin{aligned}
    \eps'(t)S(t)\eps(t)
    &
    \leq 
    \e^{-\beta (t-t_{2k+1})}\eps'(t_{2k+1})S(t_{2k+1})\eps(t_{2k+1})   
    \\
    &
    \leq \e^{-\beta (t-t_{2k+1})}\e^{-\alpha \tobs}\eps'(t_{2k})S(t_{2k})\eps(t_{2k})   
\end{aligned}
$$
We know $S(t_{2k+1})\geq  \e^{-\beta T}G_0(0,T) $, and thus with worst possible exponential decay, we get for $t\in [t_{2k+1},t_{2k+1}+\tstab)$
$$
S(t)\geq \e^{-(\beta+2a_\infty) (t-t_{2k+1}) }\e^{-\alpha T}G_0(0,T).
$$
Hence 
\begin{equation}\label{E:eps_beta}
|\eps(t)|
\leq  
\e^{a_\infty \tstab}
\e^{-\frac{\alpha}{2}(\tobs-T)}
\sqrt{\frac{\smax}{g_0}} d_0
\end{equation}
and
$$
|S^{-1}(t)C'C\varepsilon(t)|
\leq
\e^{(\beta+3 a_\infty )\tstab}\e^{-\frac{\alpha}{2}(\tobs-3T)}\sqrt{\frac{\smax}{g_0^3}}|C|^2 d_0
$$
As a consequence of \eqref{E:min_alpha_0}, we have for all $t\in [t_{2k+1},t_{2k+1}+\tstab)$
$$
m((1+2\eta)R_0)
|S^{-1}(t)C'C\varepsilon(t)|
\leq R_0.
$$
Assume there exists $t^*=\inf\{t\in[t_{2k+1},t_{2k+1}+\tstab)\mid V(\xhat(t))=(1+2\eta)R_0\}$. Then for all $t\in [t_{2k+1},t^*)$,
$$
V(\xhat(t)) -V(\xhat(t_{2k+1}))
\leq  -\int_{t_{2k+1}}^t V(\xhat( s))\diff s+R_0 (t-t_{2k+1})
$$
so that, by Gr\"onwall inequality,
$$
V(\xhat(t))\leq  R_0\left(1+2\eta+ (t-t_{2k+1})\right)\e^{-(t-t_{2k+1})}.
$$
However $\left(1+2\eta+ t-t_{2k+1}\right)\e^{-(t-t_{2k+1})}\leq 1+2\eta$ for all $t\geq t_{2k+1}$, proving that $V(\xhat(t))\leq (1+2\eta)R_0$ for all $t\in[t_{2k+1},t_{2k+1}+\tstab)$ and $t^*$ is not reached. Furthermore, with $\eta$ given by \eqref{E:def_eta}, we get that $V(\xhat(t_{2k+2}+\tstab))\leq (1-\eta)R_0$.

Finally, coming back to \eqref{E:eps_beta}, if $|\eps(t_{2k+2}+\tstab)|\leq \dist(\D((1-\eta)R_0),\D(R_0)^c) $, which is implied by \eqref{E:min_alpha_0}, then $V(x(t_{2k+2}+\tstab))\leq R_0$.
\end{proof}

As a second step in the boundedness proof, we show that on the occasions where the switching conditions is not satisfied as soon as possible, so that $t_{2k+2}>t_{2k+1}+\tstab$, then the usual behaviour of the Kalman like observer holds and the Lyapunov does not grow.

\begin{lem}\label{L:borne_beta}
We define positive constants $K_3, K_4$, depending only on the problem data, $\K\times\Khat\times\S$ and $\tstab$, as follows:
with $d_0'=\diam \left(\D((1+2\eta) R_0)\right)$ and $D'=\dist\big(\D((1-\eta)R_0),\D(R_0)^c\big)$,
$$
K_3=\frac{{D'}^2}{\smax{d_0'}^2}
\quad \text{ and } \quad
K_4=\frac{(1-\eta)^2R_0^2}{\smax{d_0'}^2m((1-\eta)R_0)^2}
$$
Assume $\beta$ large enough so that
\begin{equation}\label{E:min_beta_0}
\e^{-\beta(\tstab-T)}<K_3 \, \gmin
\quad
\text{ and }
\quad 
\e^{-\beta(\tstab-3T)}<K_4 \, \gmin^3. 
\end{equation}
Let $k$ be an integer such that $2k\in [0,\kmax)$. If at $t_{2k+1}$ we have $|\xhat(t_{2k+1})-x(t_{2k+1})|\leq  \diam(\D((1+2\eta)R_0))$, and at $t_{2k+1}+\tstab$ we have 
$V(x(t_{2k+1}+\tstab))\leq R_0$, and $V(\xhat(t_{2k+1}+\tstab))\leq (1-\eta)R_0$
then for all $t\in [t_{2k+1}+\tstab,t_{2k+2}]$,
$$
V(x(t))\leq  R_0
\quad \text{ and } \quad  
V(\xhat(t))\leq (1-\eta)R_0.
$$
\end{lem}

\begin{proof}
The lemma is trivial in the case $t_{2k+2}=t_{2k+1}+\tstab$, we assume it is not the case.
Over the interval $[t_{2k+1}+\tstab, t_{2k+2})$, the usual error bound \eqref{E:error_bound} yield
\begin{equation}\label{E:majEpsNoswitch}
|\eps(t)|
\leq 
\e^{-\frac{\beta}{2} (t-t_{2k+1})}
\sqrt{\frac{S_{\max}(t_{2k+1})}{S_{\min}(t)}} |\eps(t_{2k+1})|, \qquad \forall t\in [t_{2k+1}+\tstab, t_{2k+2}).
\end{equation}
Since $g(t)>\gmin$ over $[t_{2k+1}+\tstab, t_{2k+2})$, we have $S_{\min}(t)>\e^{-\alpha T}\gmin$. Then for all $t\in [t_{2k+1}+\tstab, t_{2k+2})$
$$
|\eps(t)|
\leq 
\e^{-\frac{\beta}{2} (t-t_{2k+1}-T)}
\sqrt{\frac{\smax}{\gmin}} |\eps(t_{2k+1})|
$$
and since $|\eps(t_{2k+1})|\leq d_0'$ by assumption,
$$
|S^{-1}C'C\eps(t)|
\leq
\e^{-\frac{\beta}{2}(\tstab- 3T)}
\sqrt{\frac{\smax}{\gmin^3}} d_0'.
$$
Assume there exists a time $t^*\in [t_{2k+1}+\tstab, t_{2k+2}]$ such that $V(\xhat(t))=(1-\eta)R_0$. By definition of the Lyapunov function $V$, we have
$$
\frac{\diff }{\diff t} V(\xhat)
\leq 
-V(\xhat)-
\frac{\partial V}{\partial x}(\xhat) \, S^{-1}C'C\eps.
$$
Assuming \eqref{E:min_beta_0} holds, this implies that at $t=t^*$, 
$$
\frac{\diff }{\diff t} V(\xhat)_{|t=t^*}
\leq 
-(1-\eta)R_0
+
\sup_{\partial \D((1-\eta)R_0)}\left|\frac{\partial V}{\partial x}\right| \e^{-\frac{\beta}{2}(\tstab- 3T)}
\sqrt{\frac{\smax}{\gmin^3}} d_0'<0.
$$
Since $V(\xhat(t_{2k+1}+\tstab))\leq (1-\eta)R_0$, this implies that $t^*=t_{2k+1}+\tstab$ is the only possible value of $t^*$ and $V(\xhat(t))<(1-\eta)R_0$ for all $t\in (t_{2k+1}+\tstab, t_{2k+2})$. This proves the first part of the statement.
In order to have $V(x(t))<R_0$, it is sufficient to have $|\varepsilon(t)|< \dist(\D((1-\eta)R_0), \D(R_0)^c)$ for all $t\in [t_{2k+1}+\tstab, t_{2k+2})$.
Coming back to \eqref{E:majEpsNoswitch}, this bound is true when \eqref{E:min_beta_0} also holds.
\end{proof}

Now we are ready to end the boundedness proof with two corollaries of the previous two lemmas.

\begin{cor}
Assume the inequalities in Lemma \ref{L:borne_alpha} and \ref{L:borne_beta} hold. Then the dynamic gain matrix $S$ is bounded above and below. We denote the below bound $0<\smin\Id\leq S(t)$.
\end{cor}

\begin{proof}
We discussed the above bound $\smax$ in the preliminaries. We focus on the lower bound.
Let $k\geq 0$. We know that $S(t_{2k+1})\geq \e^{-\alpha T}G_0(0,T)$. Then with worst possible exponential decay,
$$
S(t)\geq \e^{-(\beta+2 a_\infty)\tstab}\e^{-\alpha T}G_0(0,T), \qquad \forall t\in [t_{2k+1},t_{2k+1}+\tstab].
$$
If $g(t)>\gmin$, we do not switch right away, and we are able to say that
$$
S(t)\geq \e^{-\alpha T}\gmin \Id, \qquad \forall t\in [t_{2k+1}+\tstab,t_{2k+2}].
$$
Assuming there exists $t\geq t_{2k}+\tstab$ such that $g(t)=\gmin$, then we know that at $t_{2k+1}$, whatever the past, 
$$
S_{\min}(t_{2k+1})\geq \min
\left(
    \e^{-\alpha T}\gmin,
    \e^{-(\alpha+2 a_\infty)\tstab}\e^{-\beta T}g_0
\right).
$$
Then 
$$
S_{\min}(t)\geq \e^{-\beta T} \min
\left(
    \e^{-\alpha T}\gmin,
    \e^{-(\alpha+2 a_\infty)\tstab}\e^{-\beta T}g_0
\right), 
\qquad \forall t\in [t_{2k+1},t_{2k+1}+T].
$$
Then for times larger than $t_{2k+1}+T$, we can rely on the fact that $u=0$ on $[t_{2k+1},t_{2k+2})$, implying that 
$$
S(t)\geq \e^{-\beta T}G_0(0,T),
\qquad \forall t\in [t_{2k+1}+T,t_{2k+2}].
$$
As a result, 
$$
S_{\min}(t)\geq \e^{-\beta T} \min
\left(
    \e^{-\alpha T}\gmin,
    \e^{-(\alpha+2 a_\infty)\tstab}\e^{-\beta T}g_0
\right), 
\qquad \forall t\in [t_{2k},t_{2k+2}].
$$
Hence the existence of $\smin\Id<S(t)$ for all $t\in \R^+$, independent of $S(t_0)$.
\end{proof}

We have the following conclusion, which shows how to choose $\alpha$ and $\beta$ in the rest of the proof.

\begin{cor}
Let $K_1,K_2,K_3,K_4$ be the constants defined in Lemmas \ref{L:borne_alpha} and \ref{L:borne_beta}.
Let $\ubar{\beta}>0$ be such that
$$
\e^{-\ubar{\beta}(\tstab-T)}<K_3 \, \gmin
\quad
\text{ and }
\quad 
\e^{-\ubar{\beta}(\tstab-3T)}<K_4 \, \gmin^3 
$$
and fix $\beta>\ubar{\beta}$.
Let $\ubar{\alpha}>0$ be such that
$$
\e^{-\ubar{\alpha}(\tobs-T)}<K_1 \, g_0,
\quad
\text{ and }
\quad 
\e^{-\ubar{\alpha}(\tobs-3T)}<K_2 \, g_0^3 \, \e^{-2\beta \tstab}.
$$
Then, for all $\alpha>\ubar{\alpha}$,
all trajectories of \eqref{eq:switch} with initial conditions in $\K\times\Khat\times\S\subset\D\times\R^n\times\sym$ remain in a compact subset of $\D\times\R^n\times\sym$ over $[0,+\infty)$.
\end{cor}

Now that all parameters of the system have been fixed, it remains to investigate the attractivity and stability of the system at the target point.

\subsection{Attractivity}

\begin{lem}\label{lem:eps_conv}
The estimation error $\varepsilon(t)$ tends to $0$ as $t$ tends to infinity.
\end{lem}
\begin{proof}
Let $k\in \N$. For any $t\in [t_{2k},t_{2k+1}]$,
$$
|\eps'(t)S(t)\eps(t)|\leq \e^{-\alpha (t-t_{2k})}|\eps'(t_{2k})S(t_{2k})\eps(t_{2k})|.
$$
Likewise, for any $t\in [t_{2k+1},t_{2k+2}]$,
$$
|\eps'(t)S(t)\eps(t)|\leq \e^{-\beta (t-t_{2k+1})}|\eps'(t_{2k+1})S(t_{2k+1})\eps(t_{2k+1})|.
$$
As a consequence, with $\theta =\min(\alpha,\beta)$, we get
$$
|\eps'(t)S(t)\eps(t)|\leq \e^{-\theta t}|\eps'(0)S(0)\eps(0)|.
$$
Then $|\eps(t)|\leq \e^{-\theta t} \sqrt{\smax/\smin} |\varepsilon(0)|$, which proves the result.
\end{proof}

Now, let us prove the convergence of $(\xhat, x)$ towards $0$ by investigating the behaviour of $V(\xhat)$ and $V(x)$ over the successive observation and stabilization modes.
Recall the notations $m(R)=\displaystyle\sup_{\D(R)}\left|\dfrac{\partial V}{\partial x}\right|$, and $\bar{m}=m((1+2\eta)R_0)$.

\begin{lem}\label{L:V_small_increase}
If there exist $R\in (0,R_0)$, $k\in \N$ and $\tau\in[t_{2k}, t_{2k+1})$ such that  $V(\xhat(\tau))\leq R$ and $|S^{-1}(t)C'C\eps(t)|\leq R$ for all $t\in [\tau, t_{2k+1}]$, then $V(\xhat(t))< (1+\eta) R$ for all $t\in [\tau, t_{2k+1}]$.
\end{lem}

\begin{proof}
The assumptions imply
\[
\begin{aligned}
    |\xhat(t)-\xhat(\tau)|
    \leq \int_{\tau}^t |\dot{\xhat}(s)|\diff s
    &
    \leq \int_{\tau}^t |A(0) \xhat(s)|\diff s +(t-\tau) R
    \\
    &\leq a_0\int_{\tau}^t |\xhat(s) - \xhat(\tau)|\diff s + \left(a_0 \diam \D(R)+ R\right) (t-\tau)
\end{aligned}
\]
Then we rely on Grönwall's inequality arguments. We have
\[
    |x(t)-x(\tau)|\leq\left(a_0  \diam \D(R) +R\right)(t-\tau)\e^{ a_0 (t-\tau)}.
\]
Assume there exists $t\in(\tau,t_{2k}+\tobs)$ such that $V(\xhat(t))=(1+\eta)R$, and let $t^*=\inf\{t>\tau\mid V(\xhat(t))=(1+\eta)R\}$. Then 
\[
V(\xhat(t^*))\leq R+ 
m((1+\eta)R)
\left(a_0  \diam \D(R) + R\right)\tobs\e^{ a_0 \tobs}.
\]
However, this is impossible for any given $\tobs<\tobsmax$ by definition of  $\tobsmax$ (see \eqref{E:def_tobsmax}).
Hence $t^*$ cannot exist and $V(\xhat(t_{2k+1}))<(1+\eta)R$.
\end{proof}

\begin{lem}\label{L:V_decrease}
If there exists $k\in \N$ such that  $V(\xhat(\tau))=(1+\eta)R$ for some $R\in (0,R_0)$, $\tau\in[t_{2k+1},t_{2k+2})$, and $\bar{m} |S^{-1}(t)C'C\eps(t)|\leq R$ for all $t\geq \tau$, then 
$$
V(\xhat(t))<R (1+\eta +(t-\tau))\e^{-(t-\tau)}, \qquad \forall t \in [\tau,t_{2k+2}).
$$
In particular, $s\mapsto(1+\eta +s)\e^{-s}$ is a decreasing function over $\R_+$ and $\kappa:=(1+\eta +\tstab)\e^{-\tstab}<1$.
\end{lem}

\begin{proof}
By Lemma~\ref{L:borne_alpha}, we know $V(\xhat(t))<(1+2\eta)R_0$ on $[t_{2k+1},t_{2k+2})$. Then the assumptions imply for all $t\in [\tau,t_{2k+2})$
\[
    V(\xhat(t))-V(\xhat(\tau))
    \leq -\int_{\tau}^t V(\xhat(s))\diff s + (t-\tau) R.
\]
Then Grönwall's inequality implies for all $t\in [\tau,t_{2k+2})$
\[
    V(\xhat(t))\leq R(1+\eta +(t-\tau))\e^{-(t-\tau)}.
\]
Furthermore,
by definition of $\eta$, 
$$
\kappa=(1+\eta+\tstab)\e^{-\tstab}
=
\frac{\e^{-\tstab} \left(1+\tstab+\e^{\tstab} (\tstab+2)\right)}{\e^{\tstab}+2}
<1.
$$
Hence the statement.
\end{proof}

\begin{lem}\label{lem:R_infty_exists}
Under Assumption~\ref{ass:0_observable}, for all $T>0$, for all $\gmin>0$ such that $\gmin \Id < G_0(0,T)$, there exists $R_\infty>0$ such that $V(\xhat)_{|[0,T]} \leq R_\infty$ implies $\gram_{\lambda\circ \xhat}(0,T)>\gmin\Id$.
\end{lem}

\begin{proof}The proof follows an argument made in \cite[Section 2.4.2]{MR1862985}.
For any $T>0$, the input-to-state mapping, and therefore the Gram observability matrix, are continuous with respect to the weak-$*$ topology over $L^\infty([0,T],\R)$. Banach--Alaoglu theorem implies that closed balls for the ${L^\infty([0,T],\R)}$ norm are compact. Then the image of $\mathcal{B}(r)=\{u\in
L^\infty([0,T],\R):\|u\|_{L^\infty([0,T],\R)}\leq r\}$ by $u\mapsto G_u(0,T)$ is a compact subset positive semi-definite symmetric matrices. By assumption,
$\gram_0(0,T)>\gmin \Id$, hence
there exists $r_\infty>0$ such that for all $u\in \mathcal{B}(r_\infty)$, $\gram_u(0,T)>\gmin \Id$. We then get our statement by picking $R_\infty$ small enough so that $V(\xhat)\leq R_\infty$ implies $|\lambda(\xhat)|\leq r_\infty$.
\end{proof}

Define
$
\rho
=
\dfrac{
    \sqrt{\smin^3/\smax}
}{
    |C|^2 \bar{m}
}.
$

\begin{lem}\label{L:no_switches_R_infty}
If there exists $k\in \N$, $\tau\in [t_{2k+1}, t_{2k+2})$ such that $V(\xhat(\tau))\leq R_\infty/(1+\eta)$ and $|\varepsilon(\tau)|\leq \rho R_\infty/(1+\eta)^2$, then $V(\xhat(t))\leq R_\infty$ for all $t\in[\tau,+\infty)$. Furthermore the system never switches after $\tau+T+\tobs$ (and remains in stabilization mode, i.e., $\kmax<\infty$).
\end{lem}

\begin{proof}
Over the interval $[\tau,+\infty)$, $|\varepsilon(t)|\leq
\sqrt{\smax/\smin} \,\varepsilon(\tau)$.
By definition of $\rho$,  $|\eps(\tau)|\leq \rho R_\infty/(1+\eta)^2$ implies that 
$$
\bar{m}|S^{-1}(t)C'C\eps(t)|\leq R_\infty/(1+\eta)^2,
\quad \forall t\geq \tau.
$$
Let $t^*=\inf\{t\geq\max(\tau,t_{2k+1}+\tstab)\mid g(t)\leq\gmin\}$. If $\tau<t^*$, then Lemma~\ref{L:V_decrease} applies and $V(\xhat(t))\leq R_\infty/(1+\eta)$ for all $[\tau,t^*)$. 
Assume $t^*\in (\tau +T,+\infty)$, then that would imply that we have both $g(t^*)\leq \gmin$ and $V(\xhat)\leq R_\infty$ over $[\tau, t^*]$. This is in contradiction with Lemma~\ref{lem:R_infty_exists}. Hence either $t^*=+\infty$, $g(t)>\gmin$ for all $t\in [\tau+T,+\infty)$ and the system never switches again, or $t^*\in[\tau,\tau+T)$. In that second case,  the system switches at time $t^*=t_{2k+2}$. Then Lemma~\ref{L:V_small_increase} applies to show that $V(\xhat(t_{2k+3}))<R_\infty$ and Lemma~\ref{L:V_decrease} applies to show that $V(\xhat(t))<R_\infty$ for all $t\in [t_{2k+3},t_{2k+4})$. Then, like before, if we define $t^{**}=\inf\{t\geq t_{2k+3}+\tstab\mid g(t)\leq\gmin\}$, we know that having $t^{**}<\infty$ implies a contradiction with Lemma~\ref{lem:R_infty_exists}. Hence why $t_{2k+4}=+\infty$ and $t_{2k+3}<\tau+T+\tobs$ is the last switching time.
In both cases, Lemma~\ref{L:V_decrease} allows to conclude that $V(\xhat)<R_\infty$ over $[\tau,+\infty)$.
\end{proof}

\begin{cor}\label{cor:conv_etat}
The state $x(t)$ and observer $\xhat(t)$ both tend to the target $0\in\R^n$ as $t$ tends to $\infty$. 
Moreover, there are at most a finite number of switches, i.e., $k_{\max}<\infty$.
\end{cor}
\begin{proof}
We know that $\varepsilon\to 0$, it is sufficient to prove $\xhat\to 0$. This is achieved by proving that $V(\xhat)\to 0$. Since $\varepsilon\to 0$, there exists $\bar t$ such that for all $t\geq \bar t$, $|\varepsilon|\leq \rho R_\infty /(1+\eta)^2$. 

Under these circumstances, we can check that there are at most a finite number of switches until $V(\xhat)<R_\infty/(1+\eta)$. Indeed if the sequence $(t_k)$ is infinite, we denote $\bar k$ the first index for which $t_{2k+1}+\tstab>\bar t$. For all $k\geq \bar k $, we have by Lemmas \ref{L:V_small_increase} and \ref{L:V_decrease} that 
$$
V(\xhat(t_{2k+3}+\tstab))<\kappa V(\xhat(t_{2k+1}+\tstab)). 
$$
Since for all $k\in \N$, $V(\xhat(t_{2k+1}+\tstab))\leq R_0$, we have 
$$
V(\xhat(t_{2\bar{k}+2\ell+1 }+\tstab))<\kappa^\ell R_0 \xrightarrow[\ell \to \infty]{} 0.
$$
Hence there are at most a finite number of switches before $V(\xhat(t_{2k+1 }+\tstab))<R_\infty/(1+\eta)$ for all $k$ large enough. Then Lemma~\ref{L:no_switches_R_infty} applies with $\tau=t_{2k+1 }+\tstab$, implying that there are no more switches over $[\tau+T+\tobs,+\infty)$. Once this holds, Lemma~\ref{L:V_decrease} implies that $V(\xhat)\to 0$, and thus $\xhat\to0$.
\end{proof}

Finally, let us prove the convergence of $S$ towards $S_\infty$.

\begin{cor}\label{cor:conv_S}
The Lyapunov equation $A(0)'S_\infty + S_\infty A(0) + \beta S_\infty = C'C$ admits a unique solution $S_\infty\in\sym$.
Moreover, the dynamic gain matrix $S$ tend to $S_\infty$ as $t$ tends to $\infty$.
\end{cor}
\begin{proof}
The existence of $S_\infty$ follows from the observability of $(C, A(0))$.
Let $\tau_0>0$ be such that there are no more switches $\tau_0$, i.e, $\tau_0\geq t_{\kmax-2}$.
Then $\dot S = -A(u(t))'S - A(u(t))S-\beta S+C'C$ for all $t\geq \tau_0$.
Recall that $\beta>2\sqrt{\tr(A(u(t))'A(u(t)))}$ (see Section~\ref{sec:prel}).
Computing $\frac{\diff}{\diff t}\tr((S(t)-S_\infty)^2)$ and using Cauchy-Schwartz and Young's inequalities, we get 
that there exist two positive constant $\nu_1$ and $\nu_2$ such that
\begin{align*}
\frac{\diff}{\diff t}\tr((S(t)-S_\infty)^2)
\leq -\nu_1 \tr((S(t)-S_\infty)^2) + \nu_2 \tr((A(u(t)) - A(0))'(A(u(t)) - A(0)))
\end{align*}
Since $\xhat(t)\to0$ as $t\to+\infty$ according to Corollary~\ref{cor:conv_etat}, and $\lambda$ and $A$ are continuous,
$\tr((A(u(t)) - A(0))'(A(u(t)) - A(0))) \to 0$.
Hence, $\tr((S(t)-S_\infty)^2) \to 0$ as $t\to+\infty$, i.e., $S(t)\to S_\infty$.
\end{proof}

\subsection[Local stability of state and observer]{Local stability of $(x, \hat x)$}

\begin{cor}\label{cor:stab}
    For all $R>0$, there exists $r>0$ such that for all $\tau\in\R_+$, if $|x(\tau)|<r$ and $|\xhat(\tau)|<r$
    then $|x(t)|<R$ and $|\xhat(t)|<R$ for all $t\geq \tau$.
\end{cor}
\begin{proof}
Recall that $\eps = \xhat-x$ and $V$ is a Lyapunov function.
Hence, equivalently, we prove that for all $R_{\xhat}>0$ and all $R_\eps>0$, there exist $r_{\xhat}>0$ and $r_\eps>0$ such that for all $\tau\in\R_+$, if $V(\xhat(\tau))<r_\eps$ and $|\eps(\tau)|<r_{\xhat}$,
then $V(\xhat(t))<R_{\xhat}$ and $|\eps(t)|<R_\eps$ for all $t\geq \tau$.
According to the proof of Lemma~\ref{lem:eps_conv}, $|\eps(t)|\leq \sqrt{\smax/\smin} |\eps(\tau)|$. Hence, we choose $r_\eps \leq R_\eps\sqrt{\smin/\smax}$.
Remark that, according to Lemma~\ref{lem:R_infty_exists}, any $\tilde R_\infty<R_\infty$ is also such that $V(\xhat)\leq \tilde R_\infty$ implies $\gram_{\lambda\circ \xhat}(0,T)>\gmin\Id$.
Hence, with no loss of generality, we can suppose $R_\infty=R_{\xhat}$.
Then, if $\tau\in[t_{2k+1}, t_{2k+2})$ for some $k$,
then the result holds if $r_{\xhat}\leq R_{\xhat}/(1+\eta)$ and $r_\eps\leq \rho R_{\xhat}/(1+\eta)^2$
according to Lemma~\ref{L:no_switches_R_infty}.
Finally, if $\tau\in[t_{2k}, t_{2k+1})$ for some $k$,
then $|\eps(t)|< r_\eps\sqrt{\smax/\smin}$ for all $t\geq \tau$ by Lemma~\ref{lem:eps_conv}. 
Hence $|S^{-1}C'C\eps(t)|<r_\eps \sqrt{\smax/\smin^3} |C|^2$.
According to Lemma~\ref{L:V_small_increase}, it implies that
$V(\xhat(t_{2k+1})) < (1+\eta)\max(r_{\xhat}, r_\eps \sqrt{\smax/\smin^3} |C|^2)$.
Hence the result holds if
$(1+\eta)\max(r_{\xhat}, r_\eps \sqrt{\smax/\smin^3} |C|^2)<R_{\xhat}/(1+\eta)$ and $r_\eps\sqrt{\smax/\smin}<\rho R_{\xhat}/(1+\eta)^2$
according to Lemma~\ref{L:no_switches_R_infty}.
Hence, in any case, $V(\xhat(t))<R_{\xhat}$ and $|\eps(t)|<R_\eps$ hold for all $t\geq \tau$ if $r_\eps \leq \min(R_\eps, R_{\xhat} \rho\min(1, \bar{m})/(1+\eta)^2)\sqrt{\smin/\smax}$ and $r_{\xhat} \leq R_{\xhat}/(1+\eta)^2$.
\end{proof}

The combination of Corollary~\ref{cor:conv_etat}, \ref{cor:conv_S} and \ref{cor:stab} conclude the proof of Theorem~\ref{th:main}.

\section{Numerical simulations}

We propose a numerical simulation of the stabilization strategy of Theorem~\ref{th:main} (numerical implementation can be found in repository \cite{git}).
In dimension $n=2$, we choose $A(u) = \begin{pmatrix}
0&1+u\\-(1+u)&0
\end{pmatrix}$, $B(u)=\begin{pmatrix}
u\\0
\end{pmatrix}$ and $C = \begin{pmatrix}
0&1
\end{pmatrix}$ for all $u\in\R$.
Note that the pair $(C, A(0))$ is observable, hence Assumption~\ref{ass:0_observable} is satisfied.
This system is not uniformly observable, since the pair $(C, A(-1))$ is not observable.
One can easily check that the linear feedback law $\lambda$ defined by $\lambda(x) = \begin{pmatrix}
-1&0
\end{pmatrix}x$ for all $x\in\R^2$ is globally asymptotically stabilizing by considering the Lyapunov function $V:x\mapsto |x|^2$. Hence, following the discussion below Assumption \ref{ass:stab}, one can exhibit a bounded smooth asymptotically stabilizing state feedback coinciding with $\lambda$ over an arbitrarily large compact set, hence having arbitrary large basin of attraction.

The parameters of the switched dynamic output feedback are chosen according to Table~\ref{tab:param}. Moreover, the dynamics of $G_u(t-T, t)$ given in \eqref{eq:gram} is replaced by a stable version where $A(u)$ is replaced by $A(u)+\gamma\Id$ with $\gamma=10$, in order to ensure robustness with respect to numerical integration errors (the modification has no effects on the positivity of the Gram observability matrix).
The initial conditions are $x_0 = (-10, 0)$, $\hat x_0 = (-15, 5)$ and $S_0 = \Id$. This choice leads the system to cross the observability singularity $u=-1$ when using the control $u=\lambda(\hat x)$, i.e., $\hat x_1=1$. The Cauchy problem system is solved by means of a Runge-Kutta (2,3) method, taking into account the delayed term appearing in \eqref{eq:res}--\eqref{eq:gram}.
The resulting trajectory $(x, \xhat)$ is plotted in Figure~\ref{fig:traj} for $t\in[0, 50]$.
The switching times between observation and stabilization modes are emphasized on all figures.
In Figure~\ref{fig:norm}, the evolution of $|x|^2$ and $|\eps|^2$ is plotted.
The value of the lowest eigenvalue of the observability Gramian is shown in Figure~\ref{fig:gram}, as well as a square signal whose value represents the mode of the system  (starting with observability).
As expected by Theorem~\ref{th:main}, both the estimation error and the state's norm converge to zero. Note that the system only switches a finite number of times (10 times) before 
entering a stabilization mode it doesn't leave.
Observation modes start when the observability Gramian is too low, and make it increase. This helps the system to cross the observability singularity $\hat x_1=1$.

\begin{table}[ht!]
    \centering
    \begin{tabular}{|c|c|c|}
        \hline
        $\tstab = 3$ & $\tobs = 2$& $T = 1$\\
        \hline
        $\gmin = 5*10^{-4}$ &
        $\beta = 1$ & $\alpha = 1$\\
        \hline
    \end{tabular}
    \caption{Parameters of the numerical simulation}
    \label{tab:param}
\end{table}

\begin{figure}[ht!]
    \centering
    \includegraphics{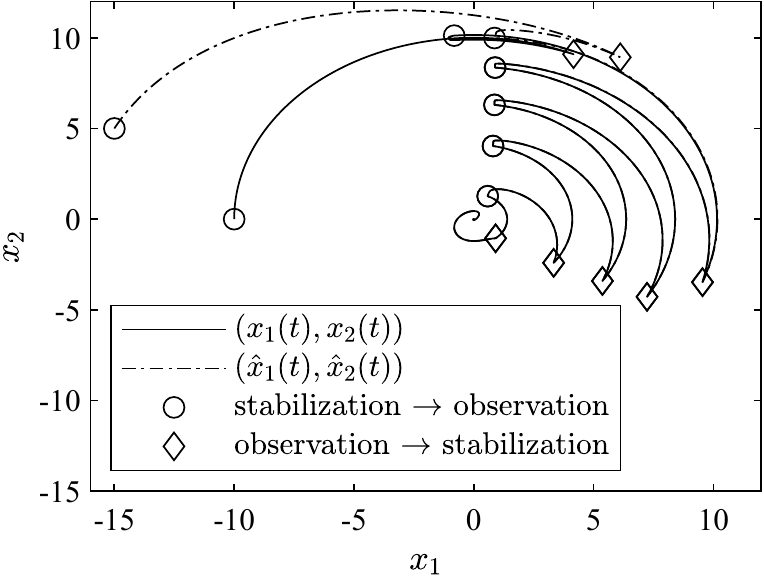}
    \caption{Trajectory of the closed-loop system. After 10 switches from one mode to an other, the system enters in a final stabilization mode.}
    \label{fig:traj}
\end{figure}

\begin{figure}[ht!]
    \centering
    \includegraphics{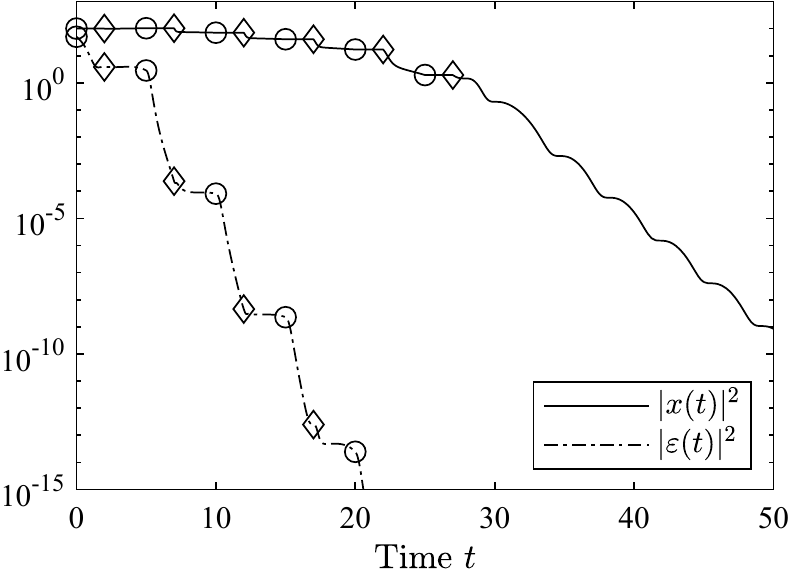}
    \caption{Evolution of the norms of the state and estimation error}
    \label{fig:norm}
\end{figure}

\begin{figure}[ht!]
    \centering
    \includegraphics{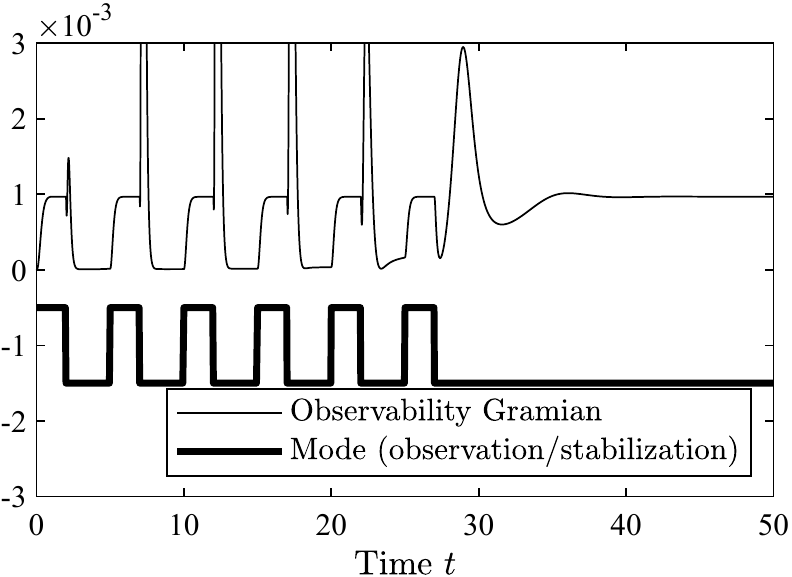}
    \caption{Evolution of the lowest eigenvalue of the observability Gramian. In bold, the squared signal represent the mode: the system is in observation mode when the signal is at its maximum, and in stabilization mode when it is at its minimum.}
    \label{fig:gram}
\end{figure}

\section{Conclusion}

In this paper, we have proposed a new output feedback stabilization strategy for non-uniformly observable state-affine systems based on switches between observation and stabilization modes.
Our main result states
that for any compact set of initial conditions,
and for well-chosen constants of time, observer gains and switching conditions, 
trajectories of the resulting closed-loop system converge to the target point and present some stable behavior.
Moreover, we show that the system actually switches a finite number of times, before entering in a final stabilization mode.
Numerical simulations on a harmonic oscillator with time-dependent speed confirm this theoretical behavior.

In future works, two important questions remain be tackled. First, stability of the closed-loop system with respect to measurement noise or perturbations of the Gram observability matrix used in the switching condition could be explored and numerically tested.
Then, it would be interesting to investigate the switching condition (currently depending on the Gram observability matrix) and to see what general class of conditions could be used instead.

\bibliographystyle{alpha}
\bibliography{references.bib}

\end{document}